\documentclass[a4paper,11pt]{article}
\usepackage{amsmath,amssymb,amsthm}
\title{\bf On the development of Bohr's phenomenon in the context of Quaternionic analysis and related problems\footnote{This article will be published in the Proceedings of the 17-th ICFIDCAA 2009 in Ho Chi Minh City.}}
\author{K. G\"urlebeck\thanks{Bauhaus-Universit\"at Weimar, Institut f\"ur Mathematik/Physik, Coudraystr. 13B, D-99421 Weimar, Germany. Email: klaus.guerlebeck@uni-weimar.de} \, and 
J. Morais\thanks{Bauhaus-Universit\"at Weimar, Institut f\"ur Mathematik/Physik, Coudraystr. 13B, D-99421 Weimar, Germany. Email: jmorais@mat.ua.pt}}
\date{}
\newtheorem{Theorem}{Theorem}[section]
\newtheorem{Lemma}{Lemma}[section]
\newtheorem{Definition}{Definition}[section]
\newtheorem{Remark}{Remark}[section]
\newtheorem{Proposition}{Proposition}[section]

\newtheorem{Corollary}{Corollary}[section]
\usepackage{amsfonts}
\begin{document}
\maketitle
\begin{abstract}
The Bohr theorem states that any function $f(z) = \sum_{n=0}^{\infty} a_{n} z^{n}$, analytic and bounded in the open unit disk, obeys the inequality $\sum_{n=0}^{\infty} |a_{n}| |z|^{n} < 1$ in the open disk of radius $1/3$, the so-called Bohr radius. Moreover, the value $1/3$ cannot be improved. In this paper we review some results related to this theorem for the three-dimensional Euclidean space in the setting of quaternionic analysis. The existing results for the Bohr radius will be improved and also some estimates for the hypercomplex derivative of a monogenic function by the norm of the function will be proved.
\end{abstract}


\noindent {\bf MSC Subject-Classification}: 30G35

\section{Introduction}

Quaternionic analysis can be regarded as a higher-dimensional generalization of the classical theory of holomorphic functions of one complex variable. It is particularly adequate to treat three- and four-dimensional structures. Meanwhile this theory has developed into an extensive mathematical theory having connections with boundary value problems and partial differential equations theory or other fields of physics and engineering (see e.g. \cite{GS1989,GS1997,KS1996,K2003,ShapiroVasilevski11995,ShapiroVasilevski21995} and elsewhere).

It is truly uncommon that a paper that has been set aside for almost a century finds its way back to scientific spotlight. Yet this is exactly what happened in the last decade to the short paper of H. Bohr that was published in 1914 \cite{HBohr1914}. As a matter of fact, in the last decade much effort has been put on the treatment of multi-dimensional analogues and other generalizations of Bohr's theorem. We highlight the contributions by S. Dineen and R. Timoney \cite{DineenTimoney1991}, H. Boas and D. Khavinson \cite{HD1997}, L. Aizenberg \cite{Aizenberg2000,Aizenberg2005}, L. Aizenberg, A. Aytuna, and P. Djakov \cite{AAD2000}, L. Aizenberg and N. Tarkhanov \cite{AAD2001}, V. Paulsen, G. Popescu, and D. Singh \cite{VGD2002}, C. Beneteau, A. Dahlner, and D. Khavinson \cite{BDK2004}, H. Kaptanoglu and N. Sadik \cite{KaptanogluSadik2005}, H. Kaptanoglu \cite{Kaptanoglu2006}, among others. The reference list does not claim to be complete and further references can be found in \cite{Kaptanoglu2006} and in the book \cite{KresinVladimir2007}. Note that first results on a higher dimensional counterpart of Bohr's theorem were recently obtained in the context of quaternionic analysis in \cite{GueJoao22007,GueJoaoBohr2} and \cite{JoaoThesis2009}. These first results show the possibility to generalize this theorem for monogenic functions in the unit ball of the Euclidean space $\mathbb{R}^3$ but the quality of the results was not satisfying. The values of the Bohr radius were very weak and the related inequality could only be proved in balls of radius $1/2$.

An effective way of proving Bohr's phenomenon in the complex case usually needs to establish some estimates of the Fourier coefficients of the holomorphic function by the first Fourier coefficient. This idea was also applied in the three-dimensional case in some recent papers by using properties of a convenient system of homogeneous monogenic polynomials, deeply exploited in \cite{JoaoThesis2009}; but see also Lemma \ref{Fouriercoeff} below. Another key idea in \cite{AAD2000} and \cite{AAD2001} was to use functions with positive real part and obtain a related inequality. This method has similarities with the well known Borel-Carath\'eodory's inequality. In fact, Borel-Carath\'eodory-type inequalities for Laurent series and multiple power series were essential ingredients in establishing the Bohr phenomenon in \cite{KaptanogluSadik2005} and \cite{Aizenberg2005}. Here, and analogously to the original formulation of Bohr's theorem, we consider the first approach. The present paper is restricted to the case of monogenic functions defined in the unit ball of $\mathbb{R}^3$ with values in the reduced quaternions (identified with $\mathbb{R}^3$). This class of functions coincides with the solutions of the well known Riesz system and shows more analogies to complex holomorphic functions than the more general class of quaternion-valued monogenic functions.

The paper is organized as follows. After a short Section 2 on notations, we present in Section 3 an orthonormal polynomial system of homogeneous monogenic polynomials defined in $\mathbb{R}^3$ with values in the reduced quaternions. Based on this system we find better estimates for the Bohr radius than in \cite{GueJoao22007} and \cite{GueJoaoBohr2}. Moreover, we use the results in Section 5 to establish estimates of the Euclidean norm of the (hypercomplex) derivative of a monogenic function $f$ by its supremum norm in the unit ball and $f(0)$.

\section{Basic notions}

In what follows we will work in $\mathbb{H}$, the skew field of quaternions. This means we can write each element $\textbf{a} \in \mathbb{H}$ in the form
\begin{eqnarray*}
\textbf{a}=a_0 + a_1 \textbf{e}_1+a_2 \textbf{e}_2+a_3 \textbf{e}_3, \, a_i \in \mathbb{R}, \, i=0,1,2,3
\end{eqnarray*}
where the imaginary units $\textbf{e}_i$ ($i=1,2,3$) stand for the elements of the basis of $\mathbb{H}$, subject to the multiplication rules
\begin{eqnarray*} 
&& \textbf{e}_1^2 = \textbf{e}_2^2 = \textbf{e}_3^2 = -1, \\
&& \textbf{e}_1 \textbf{e}_2 = \textbf{e}_3 = - \textbf{e}_2 \textbf{e}_1, \,\,\,\,\,
\textbf{e}_2 \textbf{e}_3 = \textbf{e}_1 = - \textbf{e}_3 \textbf{e}_2, \,\,\,\,\,
\textbf{e}_3 \textbf{e}_1 = \textbf{e}_2 = - \textbf{e}_1 \textbf{e}_3 .
\end{eqnarray*}
In this way the quaternionic algebra arises as a natural extension of the complex field $\mathbb{C}$. We further denote by $\mathcal{A} := {\rm span}_{\mathbb{R}}\{1,\mathbf{e}_1,\mathbf{e}_2\} \subset \mathbb{H}$ the space of reduced quaternions $\mathbf{x} := x_0 + x_1 \textbf{e}_1 + x_2 \textbf{e}_2$, which will be identified with $\mathbb{R}^3$. As a matter of fact, throughout the text we will often use the symbol $x$ to represent a point in $\mathbb{R}^3$ and $\mathbf{x}$ to represent the corresponding reduced quaternion. Also, we emphasize that $\mathcal{A}$ is a real vectorial subspace, but not a sub-algebra, of $\mathbb{H}$. The scalar of $\mathbf{x}$ is defined by $\mathbf{Sc}(\mathbf{x})=x_0$ and the vector part by $\textbf{Vec}(\mathbf{x}):=\underline{x}=x_1 \mathbf{e}_1 + x_2 \mathbf{e}_2$. Like in the complex case, the conjugate of $\mathbf{x}$ is the reduced quaternion $\overline{\mathbf{x}} = x_0 - x_1 \mathbf{e}_1 - x_2 \mathbf{e}_2$. The norm of $\mathbf{x}$ is given by $|\mathbf{x}| = \sqrt{\mathbf{x} \overline{\mathbf{x}}}=\sqrt{x_0^2+x_1^2+x_2^2}$, and coincides with its corresponding Euclidean norm, as a vector in $\mathbb{R}^3$.

\medskip

Let $\Omega$ be an open subset of $\mathbb{R}^3$ with a piecewise smooth boundary. It is evident that a reduced quaternion-valued function or, briefly, an $\mathcal{A}$-valued function is a mapping $\mathbf{f} : \Omega \longrightarrow \mathcal{A}$ which can be written as
\begin{eqnarray*}
\mathbf{f}(x) = [\mathbf{f}(x)]_0 + [\mathbf{f}(x)]_1 \textbf{e}_1 + [\mathbf{f}(x)]_2 \textbf{e}_2, \,\,\, x \in \Omega,
\end{eqnarray*}
where $[\mathbf{f}]_i$ $(i=0,1,2)$ are real-valued functions defined in $\Omega$. Properties such as continuity, differentiability, 
integrability, and so on, are ascribed coordinate-wise. We further introduce the real-linear Hilbert space of square integrable $\mathcal{A}$-valued functions defined in $\Omega$, that we denote by $L_2(\Omega;\mathcal{A};\mathbb{R})$. The real-valued inner product is defined by
\begin{eqnarray} \label{InnerProduct}
<\mathbf{f},\mathbf{g}>_{L_2(\Omega;\mathcal{A};\mathbb{R})} \, = \int_{\Omega}\;{\textbf{Sc}}({\overline{\mathbf{f}}\,\mathbf{g}) \, dV} \,,
\end{eqnarray}
where $dV$ denotes the Lebesgue measure in $\mathbb{R}^3$. This $\mathbb{R}$-valued inner product appeared for example in \cite{Davis1963} in the context of complex vector spaces, in \cite{GS1989} for spaces of quaternion-valued functions and it was also considered in \cite{BDS1982} for spaces of Clifford algebra-valued functions.

For continuously real-differentiable functions $\mathbf{f}:\Omega \longrightarrow \mathcal{A}$, the differential operator
\begin{eqnarray*}
D = \frac{\partial}{\partial x_0} + \textbf{e}_1\,\frac{\partial}{\partial x_1} + \textbf{e}_2\,\frac{\partial}{\partial x_2}
\end{eqnarray*}
is called generalized Cauchy-Riemann operator on $\mathbb{R}^3$ and it plays a central role in the sequel. In fact, note that
\begin{eqnarray*}
\Delta_3 \mathbf{f} = D \overline{D} \mathbf{f} = \overline{D} D \mathbf{f}, 
\end{eqnarray*}
whenever $\mathbf{f} \in C^2$, where $\Delta_3$ is the $3$-dimen\-sional Laplace operator and
\begin{eqnarray*}
\overline{D} = \frac{\partial}{\partial x_0} - \textbf{e}_1\,\frac{\partial}{\partial x_1} - \textbf{e}_2\,\frac{\partial}{\partial x_2}
\end{eqnarray*}
is the conjugate generalized Cauchy-Riemann operator.

The main objects in quaternionic analysis are the so-called monogenic functions which are the null-solutions of the operator $D$. More precisely, a continuously real-differentiable $\mathcal{A}$-valued function $\mathbf{f}$ in $\Omega$ is called monogenic in $\Omega$ if it satisfies the equation $D\mathbf{f}=0$ in $\Omega$.

\begin{Definition} {\rm (Hypercomplex derivative, see \cite{GueMal1999,MitelmanShapiro1995,Sud1979})}
Let $\mathbf{f}$ be an $\mathcal{A}$-valued monogenic function in $\Omega$. The expression $(\frac{1}{2} \overline{D}) \mathbf{f}$ is called hypercomplex derivative of $\mathbf{f}$ in $\Omega$.
\end{Definition}

\begin{Definition} {\rm (Hyperholomorphic constant)}
An $\mathcal{A}$-valued monogenic function $\mathbf{f}$ with an identically vanishing hypercomplex derivative is called hyperholomorphic constant.
\end{Definition}

\begin{Remark}
We observe that any monogenic $\mathcal{A}$-valued function is two-sided monogenic. This means that it satisfies simultaneously the equations $D\mathbf{f}=\mathbf{f}D=0$, which are equivalent to the system
\begin{eqnarray*}
(R) \,\,\, \left\{ \begin{array} {ccc}
{\rm div} \hspace{0.1cm} \overline{\mathbf{f}} &=& 0 \\ 
{\rm curl} \hspace{0.1cm} \overline{\mathbf{f}} &=& 0
\end{array}\right..
\end{eqnarray*}

As is well known, the system $(R)$ is called Riesz system \cite{Riesz1958}. It clearly generalizes the classical Cauchy-Riemann system for holomorphic functions in the complex plane.
\end{Remark}

Following \cite{Leutwiler2001}, the solutions of the system (R) are called (R)-solutions. The subspace $L_2(\Omega; \mathcal{A}; \mathbb{R})\cap \ker D$ of polynomial (R)-solutions of degree $n$ is denoted by $\mathcal{M}^+(\Omega;\mathcal{A};n)$. In \cite{Leutwiler2001}, it is shown that the space $\mathcal{M}^+(\Omega;\mathcal{A};n)$ has dimension $2n+3$. We denote further by $\mathcal{M}^+(\Omega;\mathcal{A}):=L_2(\Omega;\mathcal{A};\mathbb{R})$ the space of square integrable $\mathcal{A}$-valued monogenic functions defined in $\Omega$ and by $B := B_1(0)$ the unit ball in $ \mathbb{R}^3$ centered at the origin.

\section{Special homogeneous monogenic polynomials as solutions of the Riesz system in $\mathbb{R}^3$}

It has been shown in \cite{GueJoaoComplexVariables} that there exists a set of special $\mathcal{A}$-valued monogenic polynomials
\begin{eqnarray} \label{HMP}
\{ \mathbf{X}^{l,\dagger}_n, ~ \mathbf{Y}^{m,\dagger}_n : \, l=0,\ldots,n+1, \, m=1,\ldots,n+1 \},
\end{eqnarray}
which is preserving almost all properties of the powers of $z$ with positive exponent as we will see in the following list. In particular, it is shown that the system is analogous to the complex case of the Fourier exponential functions $\{e^{i n \theta}\}_{n \geq 0}$ on the unit circle and constitutes an extension of the role of the well known Legendre and Chebyshev polynomials. The explicit expressions of the mentioned polynomials are given shortly by
\begin{eqnarray*}
&& \mathbf{X}^{l,\dagger}_n(x) \, := \, |x|^n \left\{\frac{(n+l+1)}{2}  P^l_n \left(\frac{x_0}{|x|}\right) \, 
T_l \left(\frac{x_1}{\sqrt{x_1^2+x_2^2}}\right) \right. \\
&+& \frac{1}{4} P^{l+1}_n \left(\frac{x_0}{|x|}\right) \left[ T_{l+1} \left(\frac{x_1}{\sqrt{x_1^2+x_2^2}}\right) \mathbf{e}_1 
+ \frac{x_2}{\sqrt{x_1^2+x_2^2}} \, U_l \left(\frac{x_1}{\sqrt{x_1^2+x_2^2}}\right) \mathbf{e}_2 \right] \\
&+& \left. \frac{1}{4} (n+l+1)(n+l) P^{l-1}_n \left(\frac{x_0}{|x|}\right) \left[ 
- T_{l-1} \left(\frac{x_1}{\sqrt{x_1^2+x_2^2}}\right) \mathbf{e}_1\right.\right.\\
&+& \left.\left.\frac{x_2}{\sqrt{x_1^2+x_2^2}} \, U_{l-2} \left(\frac{x_1}{\sqrt{x_1^2+x_2^2}}\right) \mathbf{e}_2 \right] \right\}
\end{eqnarray*}
and
\begin{eqnarray*}
&& \mathbf{Y}^{m,\dagger}_n(x) \, := \, |x|^n \left\{\frac{(n+m+1)}{2}  P^m_n \left(\frac{x_0}{|x|}\right) \frac{x_2}{\sqrt{x_1^2+x_2^2}} \, U_{m-1} \left(\frac{x_1}{\sqrt{x_1^2+x_2^2}}\right) \right. \\
&+& \frac{1}{4} P^{m+1}_n \left(\frac{x_0}{|x|}\right) \left[ \frac{x_2}{\sqrt{x_1^2+x_2^2}} \, 
U_m \left(\frac{x_1}{\sqrt{x_1^2+x_2^2}}\right) \mathbf{e}_1 - T_{m+1} \left(\frac{x_1}{\sqrt{x_1^2+x_2^2}}\right) \mathbf{e}_2 \right] \\
&-& \left. \frac{1}{4} (n+m+1)(n+m) P^{m-1}_n \left(\frac{x_0}{|x|}\right) \left[ \frac{x_2}{\sqrt{x_1^2+x_2^2}} \, U_{m-2} \left(\frac{x_1}{\sqrt{x_1^2+x_2^2}}\right) \mathbf{e}_1\right.\right. \\ &+& 
\left.\left. T_{m-1} \left(\frac{x_1}{\sqrt{x_1^2+x_2^2}}\right) \mathbf{e}_2 \right] \right\}.
\end{eqnarray*}

Here, $P^l_n$ stands for the associated Legendre functions of degree $n$, $T_l$ and $U_l$ are the Chebyshev polynomials of the first and second kinds, respectively. For a more unified formulation we remind the reader that whenever $l=0$ the corresponding associated Legendre function $P^0_n$ coincides with the Legendre polynomial $P_n$, the associated Legendre function $P^{-1}_n$ is defined by $P^{-1}_n = - \frac{1}{n(n+1)} P^1_n$, and $P^l_n$ are the zero function for $l \geq n+1$. The fundamental references for these polynomials and their properties are \cite{JoaoThesis2009} and \cite{GueJoaoComplexVariables}. The same system was already introduced in \cite{DissCacao} using another approach but without discussing the connections to the Chebyshev polynomials.  

In detail, in \cite{JoaoThesis2009} and \cite{GueJoaoComplexVariables} the following properties have been proved:
\begin{enumerate}
\item The functions $\mathbf{X}^{l,\dagger}_n, \mathbf{Y}^{m,\dagger}_n$ are homogeneous monogenic polynomials;
\item Homogeneous monogenic polynomials of different order are orthogonal in $L_2(B;\mathcal{A};\mathbb{R})$;
\item $\mathbf{X}^{l,\dagger}_n(0) = \mathbf{Y}^{m,\dagger}_n(0) = \mathbf{0}$ for each $n= 1,\dots$;
\item The scalar parts of the basis functions form an orthogonal system;
\item The coordinates of the basis functions are mutually orthogonal in $L_2(B)$;
\item The gradients of the scalar parts of the basis functions are mutually orthogonal in $L_2(B;\mathcal{A};\mathbb{R})$;
\item The gradients of the scalar, $\mathbf{e}_1$ and $\mathbf{e}_2$-parts of the basis functions are mutually orthogonal in 
$L_2(B;\mathcal{A};\mathbb{R})$;
\item For $n \geq 1$, we have $(\frac{1}{2} \overline{D}) \mathbf{X}_n^{l,\dagger} = (n+l+1) \mathbf{X}_{n-1}^{l,\dagger}$ $(l=0,\ldots,n)$ and $(\frac{1}{2} \overline{D}) \mathbf{Y}_n^{m,\dagger} = (n+m+1) \mathbf{Y}_{n-1}^{m,\dagger}$ $(m=1,\ldots,n)$, i.e., the hypercomplex differentiation of a basis function delivers a multiple of another basis function;
\item The polynomials $\mathbf{X}^{n+1,\dagger}_n$ and $\mathbf{Y}^{n+1,\dagger}_n$ are hyperholomorphic constants.
\end{enumerate}

\medskip

The main results in Sections 4 and 5 need a series of technical preparations. As a matter of fact, in the following we will start by stating pointwise estimates of the basis polynomials (see \cite{GueJoaoComplexVariables}). We end this section by presenting an orthogonal decomposition for square integrable $\mathcal{A}$-monogenic functions, a fundamental tool on the development of the last sections.

\begin{Proposition} \label{modulusHMP}
For $n \in \mathbb{N}_0$ the homogeneous monogenic polynomials $(\ref{HMP})$ satisfy the following inequalities:
\begin{eqnarray*}
|\mathbf{X}^{l,\dagger}_n(x)| & \leq & \frac{1}{2} (n+1) \sqrt{\frac{(n+1+l)!}{(n+1-l)!}} \, |x|^n \\
|\mathbf{Y}^{m,\dagger}_n(x)| & \leq & \frac{1}{2} (n+1) \sqrt{\frac{(n+1+m)!}{(n+1-m)!}} \, |x|^n,
\end{eqnarray*}
for $l=0,\ldots,n+1$ and $m=1,\ldots,n+1.$
\end{Proposition}

\medskip

We shall denote by $\mathbf{X}_n^{0,\dagger,\ast}, \mathbf{X}_n^{m,\dagger,\ast}, \mathbf{Y}_n^{m,\dagger,\ast}$ $(m=1,\ldots,n+1)$ the normalized basis functions $\mathbf{X}_n^{0,\dagger}, \mathbf{X}_n^{m,\dagger}, \mathbf{Y}_n^{m,\dagger}$ in $L_2(B;\mathcal{A};\mathbb{R})$. From \cite{JoaoThesis2009} we take the following result.

\begin{Lemma} \label{ONSArbitraryBall}
For each $n$, the set of homogeneous monogenic polynomials
\begin{eqnarray*}
\left\{ \mathbf{X}_n^{0,\dagger,\ast}, \mathbf{X}_n^{m,\dagger,\ast}, \mathbf{Y}_n^{m,\dagger,\ast}: ~ m = 1,...,n+1 \right\}
\end{eqnarray*}
forms an orthonormal basis in $\mathcal{M}^+(B;\mathcal{A};n)$. Consequently, 
\begin{equation} \label{ONS}
\left\{\mathbf{X}_n^{0,\dagger,\ast}, \mathbf{X}_n^{m,\dagger,\ast}, \mathbf{Y}_n^{m,\dagger,\ast}, ~ m =
1,\ldots,n+1; n=0,1,\ldots \right\}
\end{equation}
is an orthonormal basis in $\mathcal{M}^+(B;\mathcal{A})$.
\end{Lemma}

Based on the complete orthonormal system $(\ref{ONS})$ the Fourier series of $\mathbf{f}$ with respect to the referred system in $\mathcal{M}^+(B;\mathcal{A};n)$ is defined by
\begin{eqnarray} \label{FourierSeries}
\mathbf{f} = \sum_{n=0}^{\infty} \left[ \mathbf{X}_n^{0,\dagger,\ast} a_n^0 + \sum_{m=1}^{n+1}
\left(\mathbf{X}_n^{m,\dagger,\ast} a_n^m + \mathbf{Y}_n^{m,\dagger,\ast} b_n^m \right) \right],
\end{eqnarray}
where for each $n \in \mathbb{N}_0$, $a_n^0, a_n^m, b_n^m \in \mathbb{R}$ $(m=1,\ldots,n+1)$ are the associated (real-valued) Fourier coefficients.

Using the fact that the polynomials $\mathbf{X}_n^{n+1,\dagger}$ and $\mathbf{Y}_n^{n+1,\dagger}$ are hyperholomorphic constants (see Property 9 above), in \cite{GueJoaoBohr2} it is proved that each $\mathcal{A}$-valued monogenic function can be decomposed in an orthogonal sum of a monogenic "main part" of the function $(\mathbf{g})$ and a hyperholomorphic constant $(\mathbf{h})$.

\begin{Lemma} {\rm (Decomposition Theorem)} \label{Tdecomposition}
A function $\mathbf{f} \in \mathcal{M}^+(B;\mathcal{A})$ can be decomposed into
\begin{eqnarray} \label{decomposition}
\mathbf{f} := \mathbf{f}(0) + \mathbf{g} + \mathbf{h}, 
\end{eqnarray}
where the functions $\mathbf{g}$ and $\mathbf{h}$ have the Fourier series
\begin{eqnarray*}
\mathbf{g}(x) &=& \sum_{n=1}^{\infty} \left(\mathbf{X}_n^{0,\dagger,\ast}(x) a_n^0 + \sum_{m=1}^{n} \left[
\mathbf{X}_n^{m,\dagger,\ast} (x) a_n^m + \mathbf{Y}_n^{m,\dagger,\ast}(x) b_n^m \right] \right) \\
\mathbf{h}(\underline{x}) &=& \sum_{n=1}^{\infty} \left[ \mathbf{X}_n^{n+1,\dagger,\ast}(\underline{x}) a_n^{n+1} 
+ \mathbf{Y}_n^{n+1,\dagger,\ast}(\underline{x}) b_n^{n+1} \right] .
\end{eqnarray*}
\end{Lemma}

\section{Bohr's phenomenon}

After the above preparations, we are now ready to study Bohr's phenomenon using the previous Fourier expansion. We begin with the simple case of monogenic functions with $\mathbf{f}(0)=\mathbf{0}$. This first version of a quaternionic Bohr type inequality was considered initially by the authors in \cite{GueJoao22007}, the Bohr radius was calculated and it was shown that for $r<0.047$ the desired inequality is satisfied. The main purpose of this first paper was to check only whether this theorem could be extended into the context of quaternionic analysis. It is interesting to note that the results on the Bohr radius depend essentially on the choice of the Fourier expansion in the proof, and consequently yield on the pointwise estimates of the homogeneous monogenic polynomials. 
The following combination of Proposition \ref{modulusHMP} and Corollary 10 in \cite{GueJoao22007} refines the value of Bohr radius in \cite{GueJoao22007}.

\begin{Corollary}
Let $\mathbf{f}$ be a square integrable ${\mathcal A}$-valued monogenic function with $\mathbf{f}(0)=\mathbf{0}$ and $|\mathbf{f}(x)|<1$ 
in $B$ and let
\begin{eqnarray*}
\sum_{n=1}^{\infty} \left[ \mathbf{X}_n^{0,\dagger,\ast} a_n^{0} 
+ \sum_{m=1}^{n+1} \left( \mathbf{X}_n^{m,\dagger,\ast} a_n^{m} + \mathbf{Y}_n^{m,\dagger,\ast} b_n^{m} \right) \right]
\end{eqnarray*}
be its Fourier expansion. Then
\begin{eqnarray*}
\sum_{n=1}^{\infty} \left| \mathbf{X}_n^{0,\dagger,\ast} a_n^{0} 
+ \sum_{m=1}^{n+1} \left( \mathbf{X}_n^{m,\dagger,\ast} a_n^{m} + \mathbf{Y}_n^{m,\dagger,\ast} b_n^{m} \right)\right|<1
\end{eqnarray*}
holds in the ball $\{x: |x|=r<0.125\}$.
\end{Corollary}
\begin{proof}
The proof is analogous to the one of Theorem 4 from \cite{GueJoao22007} with the estimates stated in Proposition \ref{modulusHMP}.
\end{proof}

Before we deal with the generalization of the latter inequality to a more general class of monogenic functions let us recall the  complex situation. The absolute value is taken from all summands of the same degree $n$. In the complex case this is also a first important step. All the considered functions with $f(0)=0$ are orthogonal to the constants. This property can be used to estimate all Fourier coefficients of a general holomorphic function with $|f(z)| \le 1$ by the first Fourier coefficient (see, e.g., \cite{VGD2002}). In the quaternionic context the previous arguments have to be modified since the set of ''constants'' is much bigger. Consequently, for our purposes we repeat here the arguments from \cite{GueJoaoBohr2} with appropriate modifications and amplifications. If we understand constants as monogenic functions which have an identically vanishing hypercomplex derivative, then it is immediately clear that the constant function and all monogenic functions which depend only on $x_1$ and $x_2$ are the hyperholomorphic constants. Moreover, if we, as in this paper, consider only $\mathcal{A}$-valued functions then a non-trivial hyperholomorphic constant can  have values only in $\mbox{span}_\mathbb{R}\{{\bf e_1, e_2}\}$. With these observations it seems to be natural to extend Bohr's phenomenon to all monogenic functions with $|\mathbf{f}(x)|<1$ in $B$ which are orthogonal to the subspace of the non-trivial hyperholomorphic constants in $L_2(B;\mathcal{A};\mathbb{R})$. This approach is also supported by the fact that it is shown in Lemma \ref{Tdecomposition} that each monogenic function can be decomposed in an orthogonal sum of a monogenic "main part" of the function and a hyperholomorphic constant.

In the following we state a collection of inequalities related to Bohr's phenomenon. All these results follow from the fundamental property of the system (\ref{HMP}) that the scalar parts of the basis functions are also orthogonal.

\begin{Lemma} \label{Fouriercoeff}
Let $\mathbf{f}$ be an $\mathcal{A}$-valued monogenic function such that $\mathbf{f}(x)-\mathbf{f}(0)$ is orthogonal to the hyperholomorphic constants with respect to the real-inner product $(\ref{InnerProduct})$ with $|\mathbf{f}(x)|<1$ in $B$. Then its Fourier coefficients with respect to $(\ref{FourierSeries})$ satisfy the following inequalities 
\begin{eqnarray*}
|a_n^{l}| &\leq& \max_{B} |\mathbf{Sc}(\mathbf{X}_n^{l,\dagger})| \frac{\|\mathbf{X}_n^{l,\dagger} \|_{L_2(B;\mathcal{A};\mathbb{R})}}{\|\mathbf{Sc}(\mathbf{X}_n^{l,\dagger}) \|_{L_2(B)}^2} \, 2 \sqrt{\frac{\pi}{3}} \left(2 \sqrt{\frac{\pi}{3}} - a_0^0\right) \\
|b_n^{m}| &\leq& \max_{B} |\mathbf{Sc}(\mathbf{Y}_n^{m,\dagger})| \frac{\|\mathbf{Y}_n^{m,\dagger} \|_{L_2(B;\mathcal{A};\mathbb{R})}}{\|\mathbf{Sc}(\mathbf{Y}_n^{m,\dagger}) \|_{L_2(B)}^2} \, 2 \sqrt{\frac{\pi}{3}} \left(2 \sqrt{\frac{\pi}{3}} - a_0^0\right),
\end{eqnarray*}
for $l=0,...,n$ and $m=1,...,n$.
\end{Lemma}

Lemma \ref{Fouriercoeff} allows us to deduce the following theorem as a refinement of Theorem 11 in \cite{GueJoaoBohr2}.

\begin{Theorem}
Let $\mathbf{f}$ be an $\mathcal{A}$-valued monogenic function such that $\mathbf{f}(x)-\mathbf{f}(0)$ is orthogonal to the hyperholomorphic constants with respect to the real-inner product $(\ref{InnerProduct})$ with $|\mathbf{f}(x)|<1$ in $B$ and let
\begin{equation*}
\sum_{n=0}^{\infty} \left[ \mathbf{X}_n^{0,\dagger,\ast} a_n^{0} + \sum_{m=1}^{n} \left( \mathbf{X}_n^{m,\dagger,\ast} a_n^{m}
+ \mathbf{Y}_n^{m,\dagger,\ast} b_n^{m} \right) \right]
\end{equation*}
be its Fourier expansion. Then
\begin{equation*}
\sum_{n=0}^{\infty} \left[ |\mathbf{X}_n^{0,\dagger,\ast}| |a_n^{0}| + \sum_{m=1}^{n} \left( |\mathbf{X}_n^{m,\dagger,\ast}| |a_n^{m}| + |\mathbf{Y}_n^{m,\dagger,\ast}| |b_n^{m}| \right) \right] < 1
\end{equation*}
holds in the ball of radius $r$, with $0 \leq r <0.026$.
\end{Theorem}
\begin{proof}
For the proof see Theorem 11 from \cite{GueJoaoBohr2} together with the estimates stated in Proposition \ref{modulusHMP}.
\end{proof}

If we compare Corollary 4.1 and Theorem 4.1 it becomes visible that we discussed two different generalizations of Bohr's theorem. In the case of Corollary 4.1 we collect all summands of degree $n$ and take then the modulus and in Theorem 4.1 we take the modulus of each single summand. Both ways generalize the complex case but lead to very different estimates of the Bohr radius.

\section{Extensions of Bohr's inequality}

The above scheme can be also applied to the hypercomplex derivative of $\mathcal{A}$-valued monogenic functions. By the next assertion we obtain an estimate for $|(\frac{1}{2} \overline{D}) \mathbf{f}|$ in terms of $|\mathbf{Sc}\{\mathbf{f}(0)\}|$ and the supremum of $|\mathbf{f}(x)|$ on $B$. A starting point for the following investigation is that $\mathcal{A}$-valued monogenic functions satisfy a maximum principle. Therefore, the function
\begin{eqnarray*}
\mathcal{M}_{\mathbf{f}}(r) = \sup_{|\xi| < r} |\mathbf{f}(\xi)|, \,\,\,\,\,\, 0 \leq r < 1
\end{eqnarray*}
is well defined, monotonically increasing whenever $\mathbf{f}$ is non-constant and continuous. We shall also introduce the notation
\begin{eqnarray*}
\mathcal{M}(\mathbf{f},r) = \max_{|x| = r} |\mathbf{f}(x)|, \,\,\,\,\,\, 0 \leq r < 1.
\end{eqnarray*}
This function of $r$ is called the maximum modulus function of $\mathbf{f}$.

\begin{Theorem} \label{EstimatesDerivatives}
Let $\mathbf{f}$ be a square integrable $\mathcal{A}$-valued monogenic function in $B$. Then, for $0\leq r< 1$ we have the inequality
\begin{eqnarray*}
\mathcal{M}\left((\frac{1}{2} \overline{D}) \mathbf{f},r\right) \, \leq \, \frac{8 (3r+1)}{(1-r)^5} \left(\mathcal{M}_{\mathbf{f}}(1) - |\mathbf{Sc}\{\mathbf{f}(0)\}| \right).
\end{eqnarray*}
\end{Theorem}
\begin{proof}
If $\mathcal{M}_{\mathbf{f}}(1)=\infty$ then the statement of the theorem is obviously satisfied. If  $\mathcal{M}_{\mathbf{f}}(1)<\infty$ then we consider $\mathbf{f}$ written as in $(\ref{decomposition})$, take into account the maximum modulus principle and obtain
\begin{eqnarray*}
\left|(\frac{1}{2} \overline{D}) \mathbf{f}(x)\right| \, = \, \left|(\frac{1}{2} \overline{D}) \mathbf{g}(x)\right|.
\end{eqnarray*}

We start by writing $\mathbf{g}$ as Fourier series
\begin{eqnarray*}
\mathbf{g} \, = \, \sum_{n=1}^{\infty} \left[ \mathbf{X}_n^{0,\dagger,\ast} a_n^0 + \sum_{m=1}^{n} \left( \mathbf{X}_n^{m,\dagger,\ast} a_n^m + \mathbf{Y}_n^{m,\dagger,\ast} b_n^m \right) \right] ,
\end{eqnarray*}
where for each $n \in \mathbb{N}_0$, $a_n^0, a_n^m, b_n^m \in \mathbb{R}$ $(m=1,...,n)$ are the associated Fourier coefficients. Since the referred series is convergent in $L_2(B)$, it converges uniformly to $\mathbf{g}$ in each compact subset of $B$. The series of all partial derivatives also converges uniformly to the corresponding partial derivatives of $\mathbf{g}$ in compact subsets of $B$. 

Applying the hypercomplex derivative $\frac{1}{2}\overline{D}$ to the series, using Property 8 of the basis polynomials, and taking the modulus it follows formally
\begin{eqnarray*}
\left|(\frac{1}{2} \overline{D}) \mathbf{g}(x) \right| &\leq& 
\frac{1}{\sqrt{2 \pi}} \sum_{n=1}^{\infty} \sqrt{2n+3} \, n |x|^{n-1} \left[ \sqrt{n+1} |a_n^0| \right. \\ 
&+& \left. \sum_{m=1}^{n-1} \sqrt{\frac{(n+1)^2-m^2}{n+1}} \left( |a_n^m| + |b_n^m| \right) \right].
\end{eqnarray*}

The main idea of the proof is to find relations between the general Fourier coefficients and $\mathcal{M}_{\mathbf{f}}(1) - |\textbf{Sc}\{\mathbf{f}(0)\}|$. By construction, the associated Fourier coefficients are real-valued and it becomes clear that the scalar part of $\mathbf{f}$ is given by
\begin{equation*}
\textbf{Sc}(\mathbf{f}) = \sum_{n=0}^{\infty} \left\{ \textbf{Sc}(\mathbf{X}_{n}^{0,\dagger,\ast}) a_n^0 
+ \sum_{m=1}^{n} \left[ \textbf{Sc}(\mathbf{X}_{n}^{m,\dagger,\ast}) a_n^m + \textbf{Sc} (\mathbf{Y}_{n}^{m,\dagger,\ast}) 
b_n^m \right] \right\}
\end{equation*}
and, since the scalar parts of the basis polynomials are homogeneous, the value of $\textbf{Sc}(\mathbf{f})$ at the origin is equal to
\begin{eqnarray*}
\textbf{Sc}(\mathbf{f})(0) &=& \sum_{n=0}^{\infty} \left\{ \textbf{Sc}(\mathbf{X}_{n}^{0,\dagger,\ast})(0) a_n^0 \right.\\ 
&+& \left. \sum_{m=1}^{n} \left[ \textbf{Sc}(\mathbf{X}_{n}^{m,\dagger,\ast})(0) a_n^m + \textbf{Sc} (\mathbf{Y}_{n}^{m,\dagger,\ast})(0) b_n^m \right] \right\} \, = \, \frac{1}{2} \sqrt{\frac{3}{\pi}} a_0^0.
\end{eqnarray*}

Without loss of generality we assume that $\textbf{Sc}(\mathbf{f})(0)$ is positive (otherwise we work with 
$-\textbf{Sc}(\mathbf{f})$). Multiplying both sides of the expression
\begin{eqnarray} \label{equation}
\textbf{Sc} (\mathcal{M}_{\mathbf{f}}(1) - \mathbf{f}) = \mathcal{M}_{\mathbf{f}}(1) - \textbf{Sc} (\mathbf{f})
\end{eqnarray}
by the homogeneous harmonic polynomials (see Property 4)
\begin{eqnarray*}\label{solidsphericalharmonics}
\{\textbf{Sc}(\mathbf{X}_n^{0,\dagger,\ast}), \textbf{Sc}(\mathbf{X}_n^{m,\dagger,\ast}), \textbf{Sc}(\mathbf{Y}_{n}^{m,\dagger,\ast}) : m=1,...,n\}
\end{eqnarray*}
and integrating over $B$, we get the desired relations. For simplicity we just present the proof for the coefficients of $\mathbf{X}_n^{0,\dagger,\ast}$, i.e, $a_n^0$. In particular, multiplying both sides of $(\ref{equation})$ by $\mathbf{Sc}(\mathbf{X}_k^{0,\dagger,\ast})$ $(k = 1,...)$ and integrating over the
ball, we obtain
\begin{eqnarray*}
 \int_{B} \mathbf{Sc} (\mathcal{M}_{\mathbf{f}}(1) - \mathbf{f}) \mathbf{Sc}(\mathbf{X}_k^{0,\dagger,\ast}) dV = - \sum_{n=0}^{\infty} a_n^0 \int_{B} \mathbf{Sc}(\mathbf{X}_n^{0,\dagger,\ast}) 
\mathbf{Sc}(\mathbf{X}_k^{0,\dagger,\ast}) dV \\
{\Longleftrightarrow} -a_k^0 = \frac{1}{\|\mathbf{Sc}(\mathbf{X}_k^{0,\dagger,\ast}) \|_{L_2(B)}^2 } 
\int_{B} \mathbf{Sc} (\mathcal{M}_{\mathbf{f}}(1) - \mathbf{f}) \mathbf{Sc}(\mathbf{X}_k^{0,\dagger,\ast}) dV.
\end{eqnarray*}

We shall relate now the integral on the right-hand side of the previous equality with $\mathcal{M}_{\mathbf{f}}(1) - |\textbf{Sc}\{\mathbf{f}(0)\}|$. Applying the modulus in the previous equation it follows
\begin{eqnarray*}
|a_k^0| &\leq& \frac{1}{\|\mathbf{Sc}(\mathbf{X}_k^{0,\dagger,\ast}) \|_{L_2(B)}^2} 
\int_{B} |\mathbf{Sc}(\mathcal{M}_{\mathbf{f}}(1) - \mathbf{f})| |\mathbf{Sc}(\mathbf{X}_k^{0,\dagger,\ast})| dV \\
&\leq& \frac{1}{\|\mathbf{Sc}(\mathbf{X}_k^{0,\dagger,\ast})\|_{L_2(B)}^2} \max_{B} 
|\mathbf{Sc}(\mathbf{X}_k^{0,\dagger,\ast})| \int_{B} \mathbf{Sc} (\mathcal{M}_{\mathbf{f}}(1) - \mathbf{f}) dV \\
&=& \max_{B} |\mathbf{Sc}(\mathbf{X}_k^{0,\dagger})| \frac{\|\mathbf{X}_k^{0,\dagger} \|_{L_2(B;\mathcal{A};\mathbb{R})}}{\|\mathbf{Sc}(\mathbf{X}_k^{0,\dagger}) \|_{L_2(B)}^2} \, \frac{4 \pi}{3} \left(\mathcal{M}_{\mathbf{f}}(1) - \mathbf{Sc}\{\mathbf{f}(0)\} \right).
\end{eqnarray*}

The last step follows from the relation
\begin{eqnarray*}
2 \sqrt{\frac{\pi}{3}} \mathcal{M}_{\mathbf{f}}(1) - a_0^0 &=& 2 \sqrt{\frac{\pi}{3}} \mathcal{M}_{\mathbf{f}}(1) \, 
- <\mathbf{X}_0^{0,\dagger,\ast},\mathbf{f} >_{L_2(B;\mathcal{A};\mathbb{R})} \\
&=& 2 \sqrt{\frac{\pi}{3}} \mathcal{M}_{\mathbf{f}}(1) \, - \int_{B} \frac{1}{2} \sqrt{\frac{3}{\pi}} ~ \mathbf{Sc}(\mathbf{f}) 
dV \\
&=& \frac{1}{2} \sqrt{\frac{3}{\pi}} \int_{B} \mathbf{Sc} (\mathcal{M}_{\mathbf{f}}(1) - \mathbf{f}) dV.
\end{eqnarray*}

Accordingly, we have found the desired relations for the coefficients $a_k^0$. The remaining coefficients $a_k^m$ and 
$b_k^m$ $(m=1,...,k)$ are obtained in a similar way. We can then state the following results:
\begin{eqnarray*}
|a_k^p| &\leq& \max_{B} |\mathbf{Sc}(\mathbf{X}_k^{p,\dagger})| \frac{\|\mathbf{X}_k^{p,\dagger} \|_{L_2(B;\mathcal{A};\mathbb{R})}}{\|\mathbf{Sc}(\mathbf{X}_k^{p,\dagger}) \|_{L_2(B)}^2} \, \frac{4 \pi}{3} \left(\mathcal{M}_{\mathbf{f}}(1) - \mathbf{Sc}\{\mathbf{f}(0)\} \right)
\end{eqnarray*}
\begin{eqnarray*}
|b_k^p| &\leq& \max_{B} |\mathbf{Sc}(\mathbf{Y}_k^{p,\dagger})| \frac{\|\mathbf{Y}_k^{p,\dagger} \|_{L_2(B;\mathcal{A};\mathbb{R})}}{\|\mathbf{Sc}(\mathbf{Y}_k^{p,\dagger}) \|_{L_2(B)}^2} \, \frac{4 \pi}{3} \left(\mathcal{M}_{\mathbf{f}}(1) - \mathbf{Sc}\{\mathbf{f}(0)\} \right),
\end{eqnarray*}
for $p=1,...,k$. In addition, the norms estimates of the basis polynomials (see \cite{DissCacao}) and of their corresponding scalar parts (see \cite{GueJoaoBohr2}) can be used for a simplification of the previous inequalities. Finally, with some extensive calculations we obtain
\begin{eqnarray*}
\left|(\frac{1}{2} \overline{D}) \mathbf{g}(x)\right| \, \leq \, \frac{4}{3} \left(\mathcal{M}_{\mathbf{f}}(1) - \mathbf{Sc}\{\mathbf{f}(0)\} \right) \sum_{n=1}^{\infty} |x|^{n-1} n^2 (n+1) (n+2).
\end{eqnarray*}
\end{proof}

Theorem \ref{EstimatesDerivatives} has a number of interesting corollaries. In the following, we state one of them.

\begin{Corollary}
Let $\tilde{\mathbf{f}}$ be a square integrable $\mathcal{A}$-valued monogenic function in $B$ orthogonal to the non-trivial hyperholomorphic constants with respect to the real-inner product $(\ref{InnerProduct})$. Then, for $0\leq r< 1$ we have the inequality
\begin{eqnarray*}
\mathcal{M}\left((\frac{1}{2} \overline{D}) \tilde{\mathbf{f}},r\right) \, \leq \, \frac{8 (3r+1)}{(1-r)^5} \left(\mathcal{M}_{\tilde{\mathbf{f}}}(1) - |\tilde{\mathbf{f}}(0)| \right).
\end{eqnarray*}
\end{Corollary}
\begin{proof}
Let $\mathbf{f}$ be written as Fourier series
\begin{eqnarray*}
\tilde{\mathbf{f}} \, = \, \sum_{n=0}^{\infty} \left[ \mathbf{X}_n^{0,\dagger,\ast} a_n^0 + \sum_{m=1}^{n} \left( \mathbf{X}_n^{m,\dagger,\ast} a_n^m + \mathbf{Y}_n^{m,\dagger,\ast} b_n^m \right) \right] ,
\end{eqnarray*}
where for each $n \in \mathbb{N}_0$, $a_n^0, a_n^m, b_n^m \in \mathbb{R}$ $(m=1,...,n)$ are the associated Fourier coefficients. At first, notice that in the previous series the sum which contains the variable $m$ runs now only from $1$ to $n$. This fact expresses the supposed orthogonality to the hyperholomorphic constants $\mathbf{X}^{n+1,\dagger}_n, \mathbf{Y}^{n+1,\dagger}_n$. Also, we recognize from the structure of the series that $\tilde{\mathbf{f}}(0)$ is real.
\end{proof}


\begin{thebibliography}{99}

\bibitem{Aizenberg2000} L. Aizenberg. \textit{Multidimensional analogues of Bohr's theorem on power series}. Proceedings of the American Mathematical Society, Vol. 128, No. 4 (2000), pp. 1147--1155.

\bibitem{AAD2000} L. Aizenberg, A. Aytuna and P. Djakov. \textit{An abstract approach to Bohr's phenomenon}. Proceedings of the American Mathematical Society, Vol. 128, No. 9 (2000), pp. 2611--2619.

\bibitem{AAD2001} L. Aizenberg and N. Tarkhanov. \textit{A Bohr phenomenon for elliptic equations}. Proceedings of the 
London Mathematical Society, Vol. 82, No. 2 (2001) pp. 385--401.

\bibitem{Aizenberg2005} L. Aizenberg. \textit{Generalization of Carath\'eodory's inequality and the Bohr radius for multidimensional power series}. Selected topics in complex analysis, Operator theory Advances and Applications, Vol. 128, Birkh\"{a}user, Basel (2005), pp. 87--94.

\bibitem{BDK2004} C. Beneteau, A. Dahlner and D. Khavinson. \textit{Remarks on the Bohr Theorem}. Computational Methods and Function Theory, Vol. 4, No. 1 (2004), pp. 1--19.

\bibitem{HD1997} H. Boas and D. Khavinson. \textit{Bohr's Power Series Theorem in several variables}. Proceedings of the American Mathematical Society, Vol. 125, No. 10 (1997), pp. 2975--2979.

\bibitem{HBohr1914} H. Bohr. \textit{A theorem concerning power series}. Proceedings of the London Mathematical Society (2) 13 (1914), pp. 1--5.

\bibitem{BDS1982} F. Brackx, R. Delanghe and F. Sommen. \textit{Clifford Analysis}. Pitman Publishing, Boston-London-Melbourne, 1982.

\bibitem{BrackxSommenAcker1994} F. Brackx, F. Sommen and N. Van Acker. \textit{Reproducing Bergman kernels in Clifford Analysis}. Complex Variables Theory and Application, Vol. 24, No. 3-4 (1994), pp. 191--204.

\bibitem{DissCacao} I. Ca{\c c}{\~a}o. \textit{Constructive Approximation by Monogenic polynomials}. Ph.D. thesis, Universidade de Aveiro, 2004.

\bibitem{Davis1963} P. Davis. \textit{Interpolation and Approximation}. Blaisdell Publishing Company, New York, 1963.

\bibitem{DineenTimoney1991} S. Dineen and R. Timoney. \textit{On a problem of H. Bohr}. Bull. Soc. Roy. Sci. Li\'ege, Vol. 60 (1991) pp. 401--404.

\bibitem{GS1989} K. G\"{u}rlebeck and W. Spr\"{o}ssig. \textit{Quaternionic Analysis and Elliptic Boundary Value Problems}. Akademie Verlag, Berlin, 1989.

\bibitem{GS1997} K. G\"{u}rlebeck and W. Spr\"{o}ssig. \textit{Quaternionic Calculus for Engineers and Physicists}. John Wiley and Sons, Chichester, 1997.

\bibitem{GueMal1999} K. G\"{u}rlebeck and H. Malonek. \textit{A hypercomplex derivative of monogenic functions in $\mathbb{R}^{n+1}$ and its applications}, Complex Variables and Elliptic Equations, Vol. 39, No. 3 (1999), pp. 199--228.

\bibitem{GueJoao22007} K. G\"{u}rlebeck and J. Morais. \textit{Bohr's Theorem for monogenic functions}. AIP Conf. Proc. 936 (2007), pp. 750--753.

\bibitem{GueJoaoBohr2} K. G\"{u}rlebeck and J. Morais. \textit{Bohr Type Theorems for Monogenic Power Series}, Computational Methods and Function Theory, Vol. 9, No. 2 (2009), pp. 633-651.

\bibitem{GueJoaoComplexVariables} K. G\"{u}rlebeck and J. Morais. \textit{Real-Part Estimates for Solutions of the Riesz System in $\mathbb{R}^3$}, submitted.

\bibitem{Kaptanoglu2006} H. Kaptanoglu. \textit{Bohr Phenomena for Laplace-Beltrami Operators}, Indagationes Mathematicae, Vol. 17, No. 3 (2006), pp. 407--423.

\bibitem{KaptanogluSadik2005} H. Kaptanoglu and N.N. Sadik. \textit{Bohr Radii of Elliptic Regions}, Russian Journal of Mathematical Physics, Vol. 12, No. 3 (2005), pp. 363--368.

\bibitem{KS1996} V. Kravchenko and M. Shapiro. \textit{Integral Representations for Spatial Models of Mathematical Physics}. Research Notes in Mathematics, Pitman Advanced Publishing Program, London, 1996.

\bibitem{K2003} V. Kravchenko. \textit{Applied quaternionic analysis. Research and Exposition in Mathematics}. Lemgo: Heldermann Verlag, Vol. 28, 2003.

\bibitem{KresinVladimir2007} G. Kresin and V. Maz'ya. \textit{Sharp Real-Part Theorems - A Unified Approach}. Lecture Notes in Mathematics, Vol. 1903, 2007.

\bibitem{Leutwiler2001} H. Leutwiler. \textit{Quaternionic analysis in $\mathbb{R}^{3}$ versus its hyperbolic modification}, Brackx, F., Chisholm, J.S.R. and Soucek, V. (ed.). NATO Science Series II. Mathematics, Physics and Chemistry, vol. 25, Kluwer Academic Publishers, Dordrecht, Boston, London, 2001, pp. 193--211.

\bibitem{MitelmanShapiro1995} I. Mitelman and M. Shapiro. \textit{Differentiation of the Martinelli-Boch\-ner integrals and the notion of hyperderivability}. Mathematische Nach\-richten 172, No. 1 (1995), pp. 211--238.

\bibitem{JoaoThesis2009} J. Morais. \textit{Approximation by homogeneous polynomial solutions of the Riesz system in $\mathbb{R}^3$}. Ph.D. thesis, Bauhaus-Universit\"at Weimar, 2009.

\bibitem{VGD2002} V.I. Paulsen, G. Popescu, and D. Singh. \textit{On Bohr's Inequality}. London Mathematical Society, Vol. 85 (2002), pp. 493--512.

\bibitem{Riesz1958} M. Riesz. \textit{Clifford numbers and spinors}. Inst. Phys. Sci. and Techn. Lect. Ser.: vol. 38. Maryland, 1958.

\bibitem{ShapiroVasilevski11995} M. Shapiro and N. L. Vasilevski. \textit{Quaternionic $\psi$-hyperholomorphic functions, singular operators and boundary value problems I}. Complex Variables, Theory Appl., 1995. 

\bibitem{ShapiroVasilevski21995} M. Shapiro and N. L. Vasilevski. \textit{Quaternionic $\psi$-hyperholomorphic functions, singular operators and boundary value problems II}. Complex Variables, Theory Appl., 1995.

\bibitem{Sud1979} A. Sudbery. \textit{Quaternionic analysis}, Math. Proc. Cambridge Phil. Soc. 85, 1979, pp. 199-225.

\end{thebibliography}
\end{document}